\documentclass[reqno, 12pt]{amsart}
\pdfoutput=1
\makeatletter
\let\origsection=\section \def\section{\@ifstar{\origsection*}{\mysection}} 
\def\mysection{\@startsection{section}{1}\z@{.7\linespacing\@plus\linespacing}{.5\linespacing}{\normalfont\scshape\centering\S}}
\makeatother        

\usepackage{amsmath,amssymb,amsthm}
\usepackage{mathrsfs}
\usepackage{mathabx}\changenotsign
\usepackage{dsfont}

\usepackage{xcolor}
\usepackage[backref]{hyperref}
\hypersetup{
    colorlinks,
    linkcolor={red!60!black},
    citecolor={green!60!black},
    urlcolor={blue!60!black}
}

\usepackage{graphicx}

\usepackage{verbatim}

\usepackage[open,openlevel=2,atend]{bookmark}

\usepackage[abbrev,msc-links,backrefs]{amsrefs} 
\usepackage{doi}

\renewcommand{\PrintDOI}[1]{\doi{#1}}

\usepackage[T1]{fontenc}
\usepackage{lmodern}
\usepackage[babel]{microtype}
\usepackage[english]{babel}

\usepackage{geometry}
\numberwithin{equation}{section}
\numberwithin{figure}{section}

\usepackage{enumitem}

\let\polishlcross=\l
\def\l{\ifmmode\ell\else\polishlcross\fi}

\def\paragraph#1{%
  \noindent\textbf{#1.}\enspace}

\let\emptyset=\varnothing
\let\setminus=\smallsetminus

\makeatletter
\def\moverlay{\mathpalette\mov@rlay}
\def\mov@rlay#1#2{\leavevmode\vtop{   \baselineskip\z@skip \lineskiplimit-\maxdimen
   \ialign{\hfil$\m@th#1##$\hfil\cr#2\crcr}}}
\newcommand{\charfusion}[3][\mathord]{
    #1{\ifx#1\mathop\vphantom{#2}\fi
        \mathpalette\mov@rlay{#2\cr#3}
      }
    \ifx#1\mathop\expandafter\displaylimits\fi}
\makeatother

\DeclareFontFamily{U}  {MnSymbolC}{}
\DeclareSymbolFont{MnSyC}         {U}  {MnSymbolC}{m}{n}
\DeclareFontShape{U}{MnSymbolC}{m}{n}{
    <-6>  MnSymbolC5
   <6-7>  MnSymbolC6
   <7-8>  MnSymbolC7
   <8-9>  MnSymbolC8
   <9-10> MnSymbolC9
  <10-12> MnSymbolC10
  <12->   MnSymbolC12}{}
\DeclareMathSymbol{\powerset}{\mathord}{MnSyC}{180}

\let\epsilon=\varepsilon

\let\rho=\varrho
\let\theta=\vartheta

\theoremstyle{plain}
\newtheorem{thm}{Theorem}[section]

\newtheorem{lemma}[thm]{Lemma}
\newtheorem{corollary}[thm]{Corollary}
\newtheorem{case}{Case}
\newtheorem{proposition}[thm]{Proposition}

\newtheorem{thm-intro}{Theorem}[]
\newtheorem*{claim*}{Claim}

\theoremstyle{definition}

\newtheorem{example}[thm]{Example}
\usepackage{thmtools}
\usepackage{thm-restate}

\usepackage{accents}

\let\phi=\varphi

\newcommand{\no}[1]{}

\def\?#1{\vadjust{\vbox to 0pt{\vss\vskip-8pt\leftline{%
     \llap{\hbox{\vbox{\pretolerance=-1
     \doublehyphendemerits=0\finalhyphendemerits=0
     \hsize16truemm\tolerance=10000\small
     \lineskip=0pt\lineskiplimit=0pt
     \rightskip=0pt plus16truemm\baselineskip8pt\noindent
     \hskip0pt        
     {\tiny #1 }\endgraf}\hskip7truemm}}}\vss}}}

\begin{document}

\author[K.~Heuer]{Karl Heuer}
\address{Karl Heuer, Institute of Software Engineering and Theoretical Computer Science, Technische Universit\"{a}t Berlin, Ernst-Reuter-Platz 7, 10587 Berlin, Germany}
\email{\tt karl.heuer@tu-berlin.de}

\author[D.~Sarikaya]{Deniz Sarikaya}
\address{Deniz Sarikaya, Department of Mathematics, University of Hamburg, Bundesstra{\ss}e 55, 20146 Hamburg, Germany}
\email{\tt deniz.sarikaya@uni-hamburg.de}

\title[]{Forcing Hamiltonicity in locally finite graphs via forbidden induced subgraphs I: nets and bulls}

\begin{abstract}
In a series of papers, of which this is the first, we study sufficient conditions for Hamiltonicity in terms of forbidden induced subgraphs and extend such results to locally finite infinite graphs.
For this we use topological circles within the Freudenthal compactification of a locally finite graph as infinite cycles.
In this paper we focus on conditions involving claws, nets and bulls as induced subgraphs.
We extend Hamiltonicity results for finite claw-free and net-free graphs by Shepherd to locally finite graphs.
Moreover, we generalise a classification of finite claw-free and net-free graphs by Shepherd to locally finite ones.
Finally, we extend to locally finite graphs a Hamiltonicity result by Ryj\'{a}\v{c}ek involving a relaxed condition of being bull-free.
\end{abstract}

\maketitle

\section{Introduction}\label{sec:Introduction}

The question whether certain graphs have a Hamilton cycle, i.e., a cycle through all vertices of the graph, and the search for necessary as well as sufficient conditions forcing Hamiltonicity is a prominent subject within graph theory.
Most results in this area focus on finite graphs, also for the reason that it is not clear what a Hamilton cycle in an infinite graphs should be.
Although two-way infinite paths might be the canonical choice for such objects, considering only them limits the class of potential Hamiltonian graphs immensely: only graphs with at most two \emph{ends} have a chance of being Hamiltonian.
The ends of a graph are the equivalence classes of \emph{rays}, i.e.~one-way infinite paths, under the relation of being inseparable by finitely many vertices.

Nevertheless, the study of Hamiltonicity has quite successfully been transferred to infinite graphs, especially for \emph{locally finite} ones, i.e.~graphs where each vertex has finite degree.
For a locally finite connected graph $G$ the Freudenthal compactification $|G|$ is considered.
This is a topological space arising from $G$ seen as a $1$-complex by adding additional points `at infinity'.
These additional points are precisely the ends of $G$ and the corresponding topology is defined in such a way that each ray converges to the end it is contained in.
For more on the space $|G|$ see~\cites{Diestel.Buch, Diestel.Arx, Freud-Equi}.
Following the topological approach by Diestel and K\"{u}hn~\cites{inf-cyc-1, inf-cyc-2}, the notion of cycles of a graph $G$ is extended to \emph{circles} in $|G|$, which are homeomorphic images of the unit circle $S^1 \subseteq \mathbb{R}^2$ in $|G|$.
This definition now allows a rather big variety of infinite cycles.
We call a circle a \textit{Hamilton circle} of $G$ if it contains all vertices of $G$.
Since Hamilton circles are closed subspaces of $|G|$, they also contain all ends of $G$.

Several Hamiltonicity results have been extended to locally finite infinite graphs so far, although not always completely, but with additional requirements~\cites{bruhn-HC, Chan.2015, Georgakopoulos.2009, Hamann.et.al.2016, Heuer.2015, Heuer.2016, Heuer.2018, Lehner.2014, Li-arx_1, Li-arx_2, Pitz.2018}.
For finite graphs many sufficient conditions guaranteeing Hamiltonicity exist which make use of global assumptions such as for example degree conditions involving the total number of vertices.
To locally finite infinite graphs, however, such conditions do not seem to be easily transferable.
For this reason we focus on sufficient conditions forcing Hamiltonicity with a local character in this series of papers, namely ones in terms of forbidden induced subgraphs.
The specific subgraphs we are focusing on in this first paper out of the series are the \emph{claw}, the \emph{net} and the \emph{bull}, which are depicted in Figure~\ref{fig:subgraphs}.
Specifically for the bull, we shall refer to its vertices $b_1, b_2$ of degree $1$ as the \emph{horns} of the bull.
In general, given two graphs $G$ and $H$ we shall call $G$ a \emph{$H$-free} graph if $G$ does not contain an induced subgraph isomorphic to~$H$.
So far sufficient conditions for Hamiltonicity in terms of forbidden induced conditions have not been analysed very much in the context of infinite graphs, although some results on claw-free graphs with additional constraints exist~\cites{Hamann.et.al.2016, Heuer.2015, Heuer.2016}.

\begin{figure}[htbp]
\centering
\includegraphics[width=13cm]{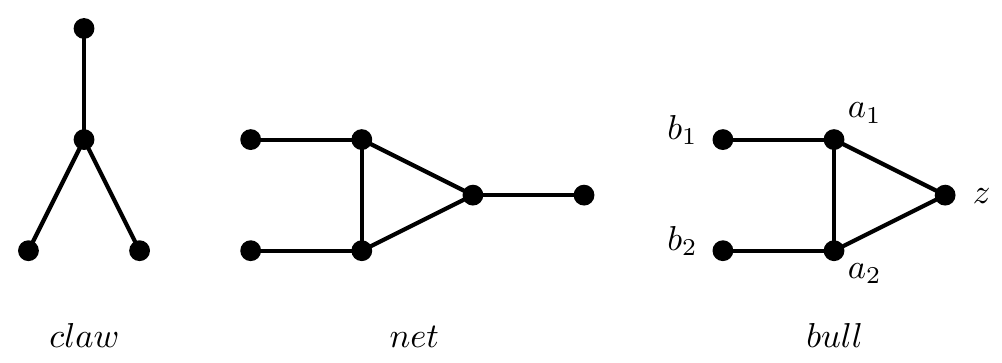}
\caption{The induced subgraphs considered in this paper.}
\label{fig:subgraphs}
\end{figure}

Our main results in this paper are centered around the following theorem by Shepherd.
In order to state it, we have to give one additional definition.
A finite graph $G$ is called \emph{$k$-leaf-connected} if $|V(G)| > k \in \mathbb{N}$ and given any vertex set $S \subseteq V(G)$ with $|S| = k$, then $G$ has a spanning tree whose set of leaves is precisely $S$.

\begin{restatable}{thm}{finShep}\cite{Shepherd.1991}*{Thm.\ 2.9}
\label{thm:finShep2.9}
Let $G$ be a finite graph. If $G$ is claw-free and net-free, then
\begin{enumerate}
	\item $G$ is connected implies $G$ has a Hamilton path. \label{thm:finShep2.9.1}
	\item $G$ is $2$-connected implies $G$ is Hamiltonian.  \label{thm:finShep2.9.2}
	\item For $k \geq 2, G$ is $(k+1)$-connected if and only if $G$ is $k$-leaf-connected.
	\label{thm:finShep2.9.3}
\end{enumerate}
\end{restatable}

Note that statement~(\ref{thm:finShep2.9.1}) and (\ref{thm:finShep2.9.2}) of Theorem~\ref{thm:finShep2.9} were already proven in~\cite{Duffus.et.al.1981}.

We shall extend all three statements of this theorem to infinite locally finite graphs.
Before that, we analyse the structure of infinite locally finite claw-free and net-free graphs, and give examples of such graphs in~Section~\ref{sec:Structure}.
Especially, we shall prove that such graphs have at most two ends.

In contrast to this, we consider locally finite graphs with potentially up to $2^{\aleph_0}$ many ends in the second paper of this series~\cite{HC_sub_2}, where we focus on the \emph{paw}, i.e. the graph obtained by attaching an edge to a triangle, and a slightly relaxed forbidden induced subgraph condition.

Regarding the first two statements of Theorem~\ref{thm:finShep2.9} we shall prove the following theorems.

\begin{restatable}{thm}{infShepOne}\label{infShep2.9.i}
For an infinite locally finite connected graph $G$ that is claw-free and net-free, precisely one of the following statements holds:
\begin{enumerate}
    \item $G$ has only one end and admits a spanning ray.
    \item $G$ has only two ends and admits a spanning double ray.
\end{enumerate}
\end{restatable}

\begin{restatable}{thm}{infShepTwo}\label{infShep2.9.ii}
Every locally finite, $2$-connected claw-free and net-free graph is Hamiltonian.
\end{restatable}

For statement~(\ref{thm:finShep2.9.3}) of Theorem~\ref{thm:finShep2.9} it might not entirely be clear at first sight how to phrase an extension of the theorem.
Note that $k$-leaf-connectivity has to be replaced in the statement.
To see this, observe that for finite graphs $3$-leaf-connectivity coincides with \emph{Hamilton connectivity}, i.e. the existence of a Hamilton path with any two previously chosen vertices as endpoints.
Within an infinite graph, an ordinary infinite path can never meet this condition.
So we define a topological analogue, called \emph{topological $k$-leaf-connectedness}, whose definition can be found in Section~\ref{subsec:top}.
We obtain the following extension of statement~(\ref{thm:finShep2.9.3}) of Theorem~\ref{thm:finShep2.9}.

\begin{restatable}{thm}{infShepThree}\label{infShep2.9.iii}
Let $G$ be a locally finite, connected, claw-free and net-free graph, and let $k \in \mathbb{N}$ satisfy $k \geq 2$.
Then $G$ is $(k+1)$-connected if and only if $G$ is topologically $k$-leaf-connected.
\end{restatable}

The key in~\cite{Shepherd.1991} to prove Theorem~\ref{thm:finShep2.9} is the following structural characterisation of the involved graphs.
In order to state this characterisation we have to give another definition first.
A graph $G$ with a vertex $v \in V(G)$ is called \emph{distance-$2$-complete centered at $v$} if $G-v$ has exactly two components and in each component $C$ and for each $i \in \mathbb{N}$, the vertices at distance $i$ from $v$ in $G[V(C) \cup \{v\}]$ induce a complete graph. 

\begin{restatable}{thm}{finShepstruc}\cite{Shepherd.1991}*{Thm 2.1}\label{finShep2.1}
A finite connected graph $G$ is claw-free and net-free if and only if for every minimal separator $S \subseteq V(G)$ and every $v \in S$, the graph $G - (S \setminus \{v \})$ is distance-$2$-complete centered at $v$.
\end{restatable}

We extend Theorem~\ref{finShep2.1} to locally finite graphs via the following result.

\begin{restatable}{thm}{infShepstruc}\label{thm:infShep2.1}
A locally finite connected graph $G$ is claw-free and net-free if and only if for every minimal finite separator $S \subseteq V(G)$ and every $v$ in $S$, the graph $G - (S \setminus \{v \})$ is distance-$2$-complete centered at $v$.
\end{restatable}

Beside Theorem~\ref{thm:finShep2.9} we shall also extend the following theorem by Ryj\'{a}\v{c}ek, which is about claw-free graphs where induced bulls may exist, but only under an additional assumption.

\begin{restatable}{thm}{finRyja}\cite{Ryjavcek.1995}*{main theorem}
\label{finRyjacek}
Let $G$ be a finite, $2$-connected, claw-free graph.
If for every induced bull $B$ in $G$ its horns have a common neighbour in $G-B$, then $G$ is Hamiltonian.
\end{restatable}

Our key to extend Theorem~\ref{finRyjacek} to locally finite graphs is the following structural result about the involved graphs.

\begin{restatable}{thm}{bullreduct}\label{bullreduct}
Let $G$ be an infinite locally finite connected claw-free graph such that for every induced bull $B \subseteq G$ the horns of $B$ have a common neighbour in $G-B$.
Then $G$ is already bull-free.
\end{restatable}

While the class of claw-free and bull-free graphs is a proper subclass of the class of claw-free graphs where the horns of every induced bull have a common neighbour outside the bull for finite graphs, they now coincide for infinite locally finite graphs.
To see that these classes differ in finite graphs, just consider a bull itself with an additional vertex only adjacent to the two horns of the bull.

Now since bull-free graphs are especially net-free, we obtain the following theorem for locally finite graphs as a corollary of Theorem~\ref{infShep2.9.ii} and Theorem~\ref{bullreduct}.

\begin{restatable}{thm}{infRyja}\label{infRyjacek}
Let $G$ be a locally finite, $2$-connected, claw-free graph.
If for every induced bull $B$ in $G$ its horns have a common neighbour in $G-B$, then $G$ is Hamiltonian.
\end{restatable}

The structure of this paper is as follows.
In Section~\ref{sec:Preliminaries} we will first introduce the needed definitions and notation.
Furthermore, we shall state all tools which we shall use to prove our main results.
In Section~\ref{sec:Structure} we shall analyse the structure of the graphs we consider in this paper and give examples of them.
We also prove Theorem~\ref{bullreduct} and Theorem~\ref{thm:infShep2.1} in that section.
Finally, we prove our main results regarding Hamiltonicity, i.e. Theorem~\ref{infShep2.9.i}, Theorem~\ref{infShep2.9.ii}, Theorem~\ref{infShep2.9.iii} and Theorem~\ref{infRyjacek} in Section~\ref{sec:Shep}.

\section{Preliminaries}\label{sec:Preliminaries}

We will follow the graph theoretical notation and use basic facts without quoting them from \cite{Diestel.Buch}, which includes especially the topological approach to locally finite graphs in \cite{Diestel.Buch}*{Ch.\ 8.5}.
For a wider survey of topological infinite graph theory, see \cite{Diestel.Arx}.

\subsection{Basic notions}\label{subsec:basic}

All graphs which are considered in this paper are undirected and simple.
In general, we do not assume a graph to be finite.
A graph is called \textit{locally finite} if every vertex has finite degree.

For the rest of this section let $G$ denote some graph.
Later in this section, however, we shall make further assumptions on $G$.

Let $X$ be a vertex set of $G$.
We denote by $G[X]$ the induced subgraph of $G$ with vertex set $X$.
For small vertex sets, we sometimes omit the set brackets, i.e.~we write $G[a,b,c]$ as a short form for $G[\{a,b,c\}]$.
We write $G-X$ for the graph $G[V \setminus X]$.
If $H$ is a subgraph of $G$ we shall write $G-H$ instead of $G-V(H)$.
Again we omit set brackets around small vertex sets, especially for singleton sets.
We briefly denote the cut $E(X, V \setminus X)$ by $\delta(X)$.
For any $i \in \mathbb{N}$ we denote by $N_i(X)$ and $N_i(v)$ the set of vertices of distance at most $i$ in $G$ from the vertex set $X$ or from a vertex $v \in V(G)$.

For two subgraphs $G_1$ and $G_2$ of $G$ we define $G_1 \cup G_2 = G[V(G_1) \cup V(G_2)]$.

Let $C$ be a cycle of $G$ and $u$ be a vertex of $C$.
We implicitly fix an orientation of the cycle and we write $u^+$ and $u^-$ for the neighbour of $u$ in $C$ in positive and negative, respectively, direction of $C$ using a fixed orientation of $C$.
Later on we will not always mention that we fix an orientation for the considered cycle using this notation.

If $v$ and $w$ are vertices of a tree $T$, then we denote by $vTw$ the unique $v$--$w$ path in $T$.

For some $k \in \mathbb{N}$, we say that a finite graph $G'$ is \emph{$k$-leaf-connected} if $|V(G')| > k$ and given any vertex set $S \subseteq V(G')$ with $|S| = k$, then $G'$ has a spanning tree $T$ whose set of leaves is precisely $S$.
Note that for any graph $G$ being $2$-leaf-connected is equivalent to being \emph{Hamilton connected}, namely that any two distinct vertices $v, w$ of $G$ are connected via a Hamilton path with $v$ and $w$ as its endpoints.

For any $v \in V(G)$ we call $G$ \emph{distance-$2$-complete centered at $v$} if $G-v$ has exactly two components and in each component $K$ and for each $i \in \mathbb{N}$, the vertices at distance $i$ from $v$ in $G[V(K) \cup \{v\}]$ induce a complete graph.

A one-way infinite path $R$ in $G$ is called a \textit{ray} of $G$.
A subgraph of a ray $R$ is called a \emph{tail} of $R$.
The unique vertex of degree $1$ of $R$ is called the \emph{start vertex} of $R$.
For a vertex $r$ on a ray $R$, we denote the tail of $R$ with start vertex $r$ by $rR$.
A two-way infinite path in $G$ is called a \emph{double ray}.

An equivalence relation can be defined on the set of all rays of $G$ by saying that two rays in $G$ are \emph{equivalent} if they cannot be separated by finitely many vertices.
It is easy to check that this defines in fact an equivalence relation.
The corresponding equivalence classes of rays under this relation are called the \textit{ends} of $G$.
We denote the sets of ends of a graph $G$ with $\Omega(G)$.
If $R \in \omega$ for some end $\omega \in \Omega(G)$, then we briefly call $R$ an \emph{$\omega$-ray}.

Note that for any end $\omega$ of $G$ and any finite vertex set $S \subseteq V(G)$ there exists a unique component $C(S, \omega)$ that contains tails of all $\omega$-rays.
We say that a finite vertex set $S \subseteq V(G)$ \emph{separates} two ends $\omega_1$ and $\omega_2$ of $G$ if $C(S, \omega_1) \neq C(S, \omega_2)$.
Note that any two different ends can be separated by a finite vertex set.

We say that a (double) ray of $G$ is \emph{geodetic} if and only if for any two vertices on the (double) ray there is no shorter path between these two vertices in $G$ than the one on the (double) ray.

Let $R$ be a ray in $G$ and $X \subseteq V(G)$ be finite.
We call $R$ \emph{distance increasing w.r.t. $X$} if $|V(R) \cap N_i(X)| = 1$ for every $i \in \mathbb{N}$.
Note that a distance increasing ray w.r.t.\ $X$ has its start vertex in $X$.

\subsection{Topological notions}\label{subsec:top}

For the rest of this section, we assume $G$ to be locally finite and connected.
A topology can be defined on $G$ together with its ends to obtain a topological space which we call $|G|$.
Note that inside $|G|$, every ray of $G$ converges to the end of $G$ it is contained in.
For a precise definition of $|G|$, see \cite{Diestel.Buch}*{Ch.\ 8.5}.
Apart from the definition of $|G|$ as in \cite{Diestel.Buch}*{Ch.\ 8.5}, there is an equivalent way of defining the topological space $|G|$, namely, by endowing $G$ with the topology of a $1$-complex and considering the Freudenthal compactification of $G$.
This connection was examined in \cite{Freud-Equi}.
For the original paper of Freudenthal about the Freudenthal compactification, see \cite{Freud}.

For a point set $X$ in $|G|$, we denote its closure in $|G|$ by $\overline{X}$ and its interior by $\mathring{X}$.
A subspace $Z$ of $|G|$ is called \textit{standard subspace} of $|G|$ if $Z = \overline{H}$ where $H$ is a subgraph of $G$.

A \textit{circle} of $G$ is the image of a homeomorphism which maps from the unit circle $S^1 \subseteq \mathbb{R}^2$ to $|G|$.
The graph $G$ is called \textit{Hamiltonian} if there exists a circle in $|G|$ which contains all vertices of $G$, and hence, by the closedness of circles, also all ends of $G$.
This circle is called a \emph{Hamilton circle} of $G$.
We note that, for finite graphs, this coincides with the usual notion of Hamiltonicity.

The image of a homeomorphism which maps from the closed real unit interval $[0, 1]$ to $|G|$ is called an \textit{arc} in $|G|$.
For an arc $\alpha$ in $|G|$, we call the images of $0$ and $1$ of the homeomorphism defining the arc, the \emph{endpoints} of the arc.
A subspace $Z$ of $|G|$ is called \textit{arc-connected} if for every two points of $Z$ there is an arc in $Z$ which has these two points as its endpoints.
Finally, an arc in $|G|$ is called a \emph{Hamilton arc} of $G$ if it contains all vertices~of~$G$.

Let $\omega$ be an end of $G$ and $Z$ be a standard subspace of $|G|$ containing $\omega$.
Then we define the \textit{degree} of $\omega$ in $Z$ as a value in $\mathbb{N} \cup \lbrace \infty \rbrace$, namely the supremum of the number of edge-disjoint arcs in $Z$ that have $\omega$ as one of their endpoint.

We make a further definition with respect to end degrees which allows us to distinguish the parity of degrees of ends when they are infinite.
This definition has been introduced by Bruhn and Stein~\cite{cycle}.
We call the degree of an end~$\omega$ of $G$ in a standard subspace $X$ of $|G|$ \textit{even} if there is a finite set $S \subseteq V(G)$ such that for every finite set $S' \subseteq V(G)$ with $S \subseteq S'$ the maximum number of edge-disjoint arcs in $X$ with $\omega$ and some $s \in S'$  as endpoints is even.
Otherwise, we call the degree of $\omega$ in~$X$~\textit{odd}.

A \emph{topological tree} of $G$ is a connected standard subspace of $\vert G \vert$ which contains no circle~of~$G$.
A topological tree of $G$ is called \emph{spanning} if it contains all vertices of $G$.
We denote such a tree also as a \emph{topological spanning tree} of $G$.
Let $T$ be a subgraph of $G$ such that $\overline{T}$ is a topological tree of $G$.
We call a point $x \in \overline{T}$ a \emph{leaf} of $\overline{T}$ if either $x \in V(G)$ and has degree $1$ in $T$ or $x \in \Omega(G)$ and has degree $1$ in $\overline{T}$.

We extend notion $k$-leaf-connectedness to locally finite connected graphs as follows.
We call $G$ \emph{topologically $k$-leaf-connected} if $|V(G)| > k$ and given any set $S \subseteq V(G) \cup \Omega(G)$ with $|S| = k$, then $G$ has a topological spanning tree $\overline{T}$ whose set of leaves is precisely~$S$.
Similarly as for finite graphs, being topologically $2$-leaf-connected coincides with the notion for locally finite connected graphs $G$ of being \emph{Hamilton connected}, i.e., for any two distinct $x, y \in V(G) \cup \Omega(G)$ there exists a Hamilton arc of $G$ that has $x$ and $y$ as its endpoints.

\subsection{Tools}\label{subsec:Tools}

In this section we introduce some basic lemmas we will use to prove our results.
We begin with stating a lemma which allows us to make slightly limited, but still very helpful compactness arguments.

\begin{lemma}\label{koenig}\cite{Diestel.Buch}*{Lemma~8.1.2 (K\H{o}nigs Infinity Lemma)}
Let $(V_i)_{i \in \mathbb{N}}$ be a sequence of disjoint non-empty finite sets, and let $G$ be a graph on their union.
Assume that for every $n > 0$ each vertex in $V_n$ has a neighbour in $V_{n-1}$.
Then $G$ contains a ray $v_0v_1 \ldots$ with $v_n \in V_n$ for all $n \in \mathbb{N}$.
\end{lemma}

An immediate consequence of Lemma~\ref{koenig} is the following proposition.

\begin{proposition}\label{ray}\cite{Diestel.Buch}*{Prop.\ 8.2.1}
Every infinite connected graph has a vertex of infinite degree or contains a ray.
\end{proposition}

Since we shall only consider locally finite connected graphs, we know by Proposition~\ref{ray} that such graphs contain a ray as soon as they are infinite.

The next lemma is also an immediate consequence of Lemma~\ref{koenig} and ensures the existence of distance increasing rays with respect to finite vertex sets in locally finite graphs.
For the sake of completeness we give a proof here.

\begin{lemma}\label{lem:dist-ray}
Let $G$ be an infinite locally finite connected graph and $X \subseteq V(G)$ be finite.
Then there exists a distance increasing ray w.r.t.~$X$.
\end{lemma}

\begin{proof}
Since $G$ is locally finite and connected, we have that each $N_i(X)$ is non-empty, but finite.
Also each vertex in $N_{i+1}(X)$ has a neighbour in $N_{i}(X)$ by definition for every $i \in \mathbb{N}$.
By Lemma~\ref{koenig} we obtain the desired ray.
\end{proof}

We state a similar lemma about the existence of geodetic double rays.

\begin{lemma}\cite{Watkins.1986}*{Thm.\ 2.2} \label{LemmaGeo}
Let $G$ be a locally finite connected graph and $\omega_1$ and $\omega_2$ two distinct ends of $G$. Then there is a geodetic double ray that is the union of an $\omega_1$-ray and an $\omega_2$-ray.
\end{lemma}

The next lemma tells us that arcs within $|G|$ have to cross a finite cut As soon as they meet both sides of the cut.

\begin{lemma}\label{jumping-arc}\cite{Diestel.Buch}*{Lemma 8.5.3 (Jumping Arc Lemma)}
Let $G$ be a locally finite connected graph and  ${F \subseteq E(G)}$ be a cut with the sides $V_1$ and $V_2$.
\begin{enumerate}
	\item \textit{If $F$ is finite, then $\overline{V_1} \cap \overline{V_2} = \emptyset$, and there is no arc in $|G| \setminus \mathring{F}$ with one endpoint in $V_1$ and the other in $V_2$.}
	\item \textit{If $F$ is infinite, then $\overline{V_1} \cap \overline{V_2} \neq \emptyset$, and there may be such an arc.}
\end{enumerate}
\end{lemma}

The following lemma gives us a combinatorial criterion when standard subspaces of $|G|$ are topologically connected.

\begin{lemma}\label{top-con}\cite{Diestel.Buch}*{Lemma 8.5.5}
If a standard subspace of $|G|$ contains an edge from every finite cut  of $G$ which meets both sides, then it is topologically connected. 
\end{lemma}

Although topological connectedness and arc-connectedness differ for general topological spaces, the do not for closed subspaces of $|G|$ as shown by the following lemma.

\begin{lemma}\label{arc_conn}\cite{path-cyc-tree}*{Thm.\ 2.6}
If $G$ is a locally finite connected graph, then every closed topologically connected subset of $|G|$ is arc-connected.
\end{lemma}

We shall make use of Lemma~\ref{top-con} and Lemma~\ref{arc_conn} to verify that ends in standard subspaces we construct have degree at least $1$ by showing that those spaces intersect every finite cut.
Similarly, the following theorem gives us a way to verify even degrees at ends.

\begin{thm}\label{cycspace} \cite{Diestel.Arx}*{Thm.\ 2.5}
Let $G$ be a locally finite connected graph. Then the following are equivalent for $D \subseteq E(G)$:
\begin{enumerate}
	\item $D$ meets every finite cut in an even number of edges.
	\item Every vertex and every end of $G$ has even degree in $\overline{D}$.
\end{enumerate}
\end{thm}

The following lemma, combined with Lemma~\ref{top-con} gives us a nearly purely combinatorial characterisation of those standard subspaces of $|G|$ which form a circle.

\begin{lemma}\cite{cycle}*{Prop.\ 3} \label{circ}
Let $C$ be a subgraph of a locally finite connected graph $G$. Then $\overline{C}$ is a circle if and only if $\overline{C}$ is topologically connected and every $v \in V(C)$ has degree $2$ in $C$ as well as every $\omega \in \Omega(G) \cap \overline{C}$ has degree $2$ in $\overline{C}$.
\end{lemma}

Our general strategy to verify that the closure $\overline{H}$ within $|G|$ of a subgraph $H$ of $G$, which we usually construct in countably many steps, is in fact a Hamilton circle of $G$ works as follows.
First we check that $H$ contains every vertex of $G$.
Then we prove that each vertex has degree $2$ in $H$, which is usually an easy task.
Now we prove that $H$ intersects every finite cut of $G$, but in even number of edges.
This already proves that $\overline{H}$ is topologically connected by Lemma~\ref{top-con} and that every end of $G$ has even degree, but least $2$ in $\overline{H}$ by Lemma~\ref{arc_conn} and Theorem~\ref{cycspace}.
By Lemma~\ref{circ} it only remains to bound the degrees of the ends, for which we use Lemma~\ref{jumping-arc} adjusted to the way we construct $H$.

\section{About the structure of the graphs considered in this paper}\label{sec:Structure}

In this section we shall analyse the structure of the graphs occur in the main results of this paper.
Furthermore, we shall give examples of the considered graphs at the end of this section.

Let us now start with a very easy observation about claw-free graphs.
The result is probably folklore and we do not give a proof here.
However, in case a proof is desired, consider for example~\cite{Heuer.2015}*{Prop.\ 3.7.}.

\begin{proposition} \label{2comp} 
Let $G$ be a connected claw-free graph and $S$ be a minimal vertex separator in $G$. Then $G-S$ has exactly two components.
\end{proposition}

Next we generalise Theorem~\ref{finShep2.1} to locally finite graphs and obtain a structural characterisation of locally finite claw-free and net-free graphs.
Note that Theorem~\ref{finShep2.1} was essential for the proof of Theorem~\ref{thm:finShep2.9} and its generalisation will also be crucial for us to extend Theorem~\ref{thm:finShep2.9} to locally finite graphs.
We restate Theorem~\ref{finShep2.1} below.
Recall that we call a graph $G$ with a vertex $v \in V(G)$ \emph{distance-$2$-complete centered at $v$} if $G-v$ has exactly two components and in each component $C$ and for each $i \in \mathbb{N}$, the vertices at distance $i$ from $v$ in $G[V(C) \cup \{v\}]$ induce a complete graph.

\finShepstruc*

Now we extend Theorem~\ref{finShep2.1} to locally finite graphs by proving Theorem~\ref{thm:infShep2.1}, which is restated below.
We would like to point out that for one implication we can actually use the same proof which was given in~\cite{Shepherd.1991} for the corresponding implication for finite graphs of Theorem~\ref{finShep2.1}.
For the sake of completeness, we include this argument here.
Let us now restate the theorem we are proving.

\infShepstruc*

\begin{proof}
Suppose first for a contradiction that there exists a locally finite connected graph $G$ for which every minimal finite separator $S \subseteq V(G)$ and every $v \in S$, the graph $G - (S \setminus \{v \})$ is distance-$2$-complete centered at $v$, but $G$ contains a claw or a net as an induced subgraph~$H$.
Let $h_1, h_2, h_3$ denote the three vertices of degree $1$ in $H$ and let us call any other vertex of $H$ a \emph{central vertex} of~$H$.
By the local finiteness of $G$ we know that $S_1 := N(h_1)$ is a finite vertex set containing a central vertex of $H^1 := H$ such that $L_1 := \{ h_1, h_2, h_3 \}$ are not contained in the same component of $G-S_1$.

Next we shall recursively define sets $S_i, H^i$ and $L_i$ until $S_i$ is a $\subseteq$-minimal vertex separator of $G$ such that the following holds for every $i \in \mathbb{N}$ with $i \geq 1$:
\begin{enumerate}
    \item $H^i$ is either an induced claw or an induced net of $G$ and $H^i \subseteq G-(S_i \setminus \{ c_i \})$.
    \item $L_i$ consists of the vertices of degree $1$ in $H^i$.
    \item $S_i$ contains a central vertex $c_i$ of $H^i$.
    \item $S_{i+1} \subsetneqq S_{i}$ if $S_i$ is no $\subseteq$-minimal vertex separator of $G$.
    \item $S_i$ separates $G$ such that $L_i$ is not contained in the same component of $G-S_i$.
\end{enumerate}

We already found suitable sets for $i=1$ above.
Now suppose $S_i, c_i, H^i$ and $L_i$ have already been defined for some $i \geq 1$.
If there exists a vertex $s \in S_i$ such that $S_i \setminus \{ s \}$ still satisfies property (5), then set $S_{i+1} := S_i \setminus \{ s \}$, the graph $H^{i+1} := H^i$, the vertex $c_{i+1} := c_i$ and $L_{i+1} := L_i$.
In this way all properties (1)-(5) are still maintained.

So suppose there does not exist such a vertex $s \in S_i$ and $S_i$ is no $\subseteq$-minimal vertex separator of $G$.
Let $K_1, K_2, K_3, \ldots, K_k$ be the components of $G-S_i$ for some $k \in \mathbb{N}$.
Since $S_i$ is no $\subseteq$-minimal vertex separator of $G$, there exists a vertex $x \in S_i$ such that $S_i \setminus \{ x \}$ is still a vertex separator of $G$.
Hence, $x$ is not adjacent to any vertex of $K_j$, for some $j \in \mathbb{N}$, say $j=1$.
Since $S_i \setminus \{ x \}$ does not satisfy property (5) with respect to $L_i$, there exist two distinct components of $G-S_i$ which are both different from $K_1$ and contain vertices from $L_i$, say these are $K_2$ and $K_3$.
Furthermore, $x$ must be adjacent to vertices in $K_2$ and $K_3$.
Now pick $v \in S_i$ which is adjacent to some vertex $x_1$ in $K_1$.
Because $S_i \setminus \{ v \}$ does not satisfy property (5) either, we know that $v$ is adjacent to vertices $x_2 \in V(K_2)$ and $x_3 \in V(K_3)$ as well.
Therefore, $H^{i+1} := G[v, x_1, x_2, x_3]$ is an induced claw of $G$.
Set $L^{i+1}$ to be the vertices of degree $1$ in $H^{i+1}$, the vertex $c_{i+1} := v$ and $S_{i+1} := S_i \setminus \{ x \}$.
Now $S_{i+1}, c_{i+1}, H^{i+1}$ and $L_{i+1}$ satisfy properties (1)-(5), which completes the recursive definition.

Let $\ell \in \mathbb{N}$ such that $S_{\ell}$ is an $\subseteq$-minimal vertex separator of $G$.
By the properties (1) and (3) we know that $S_{\ell}$ contains a central vertex $c_{\ell}$ of $H^{\ell}$ such that $H^{\ell} \subseteq G-(S_{\ell} \setminus \{ c_{\ell} \})$.
In both cases, whether $H^{\ell}$ is an induced claw or an induced net of $G$, this contradicts that $G-(S_{\ell} \setminus \{ c_{\ell} \})$ is distance-$2$-complete centered at $c_i$, which holds by assumption.

Now let us prove the converse and assume that $G$ is a locally finite connected claw- and net-free graph.
Take an arbitrary finite minimal separator $S$ of $G$ and fix some $v \in S$.
Due to Proposition~\ref{2comp}, we get that $G - S$ has exactly two components $C_1$ and $C_2$.
For each $i \in \{ 1, 2 \}$ fix a finite connected subgraph $F_i \subseteq C_i$ which contains $N(S) \cap V(C_i)$.
Since $S$ is a minimal separator, there exists a shortest $v$--$f$ path $P$ in $G - (S \setminus \{ v \}) \cap C_i$ for every $f \in F_i$ and every $i \in \{ 1, 2 \}$.
Let $n \in \mathbb{N}$ be the maximum length of all such shortest paths $P$ for all $f \in F_i$ and every $i \in \{ 1, 2 \}$.
Now set
$$G_0 := G [ S \cup \bigcup^n_{i = 0} N_i(v) ].$$

We claim that $G_0$ fulfills the antecedent of Theorem~\ref{thm:finShep2.9}.
Clearly $G_0$ is clearly connected and finite, as $G$ is locally finite.
Furthermore, since $G_0$ is an induced subgraph of the claw-free and net-free graph $G$, we know that $G_0$ is claw-free and net-free as well.
By definition, $N(S) \cup S$ is contained in $G_0$.
Hence $S$ is again a finite minimal vertex separator in $G_0$.
So Theorem~\ref{thm:finShep2.9} implies that $G_0$ is distance-$2$-complete centered at $v$.

Now suppose we have already defined $G_i$ for some $i \in \mathbb{N}$.
Then we recursively define ${G_{i+1} := G[V(G_i) \cup N(V(G_i))]}$.
By definition, $G_{i+1}$ is a finite connected graph without induced claws and without induced nets, and $S$ is a minimal separator in $G_{i+1}$.
So again by Theorem~\ref{thm:finShep2.9} we know that $G_{i+1}$ is distance-$2$-complete centered at $v$ as well.
But now we get that for every $k \in \mathbb{N}$ each distance class $N_k(v)$ induces a clique in each component of $G-(S \setminus \{ v \})$ as witnessed by $G_k$. 
\end{proof}

Although we shall not need it in order to prove our main results, we would like to add the following structural result for locally finite claw-free and bull-free graphs.
It tells us that finite minimal separators in such graphs always induce cliques.
Note that such separators always exist, since $G$ is locally finite.
Since bull-free graphs must be net-free as well, the following result combined with Theorem~\ref{thm:infShep2.1} give us very much information about the structure of infinite locally finite claw-free and bull-free graphs.

\begin{lemma}\label{lem:bull-free_clique-sep}
In an infinite locally finite claw-free and bull-free graph every finite \linebreak $\subseteq$-minimal vertex separator induces a clique. 
\end{lemma}

\begin{proof}
Let $G$ be a graph as in the statement of the lemma and let $S$ be a finite $\subseteq$-minimal vertex separator of $G$.
Suppose for a contradiction that $S$ contains two distinct vertices $u, v$ such that $uv \notin E(G)$.
Since $G$ is claw-free, we know by Proposition~\ref{2comp} that $G-S$ has precisely two components, call them $C$ and $C'$.
Furthermore, we know by Proposition~\ref{ray} that $G$ contains a ray.
So $G$ has an end $\omega$.
Since $S$ is finite, every $\omega$-ray has a tail in either $C$ or $C'$, say in $C$.
Now choose a shortest $u-v$ path $P$ in $G[V(C) \cup S]$.
Let $R = r_0r_1 \ldots$ be a distance increasing ray w.r.t.~$V(P)$ inside $G[V(C) \cup S]$ whose start vertex on $P$ lies as close to $u$ as possible.
Such a ray exists due to Lemma~\ref{lem:dist-ray}.
Next we distinguish two cases.

\begin{case}
$r_0 = u$.
\end{case}

Since $S$ is minimal, $u$ has a neighbour $w$ in $C'$.
Let $u^+$ be the vertex neighbouring $u$ on~$P$.
Note that $u^+ \neq v$ by assumption.
Since $G[w,u,u^+,r_1]$ is no induced claw and since $w$ is separated from $r_1$ and $u^+$ by the separator $S$, we know that $r_1u^+ \in E(G)$.
Now $G[u_l, u, u^+, y, y']$ is an induced bull as $r_2$ has distance $2$ from $P$ and, therefore, cannot be adjacent to $u, u^+$ or $w$.
We derived a contradiction.

\begin{case}
$r_0 \neq u$.
\end{case}

In this case we have that $x := r_0$ is an inner vertex of $P$ since $uv \notin E(G)$.
Let $x^+$ and $x^-$ denote the two neighbours of $x$ on $P$, where $x^-$ lies closer to $u$ on $P$ than $x^+$.
Consider the graph $G[x^-, x, x^+, r_1]$, which cannot be an induced claw by assumption.
Since $P$ is a shortest $u$-$v$ path, we know that $x^+x^- \notin E(G)$.
Also, we know that $x^-r_1 \notin E(G)$ as it would yield another valid choice for the ray $R$, but with a start vertex $x^-$ closer to $u$ on $P$ than $r_0$, contradicting our choice of $R$.
So the edge $x^+r_1$ exists.
Now, however, the graph $G[x^-, x, x^+, r_1, r_2]$ is an induced bull as $r_2$ has distance $2$ to $P$; a contradiction.
\end{proof}

Next let us prove Theorem~\ref{bullreduct}.
To ease the readability of the paper, let us restate the theorem here.

\bullreduct*

\begin{proof}
Suppose for a contradiction that $G$ contains an induced bull $B$.
Let $b_1$ and $b_2$ denote the horns of $B$, and let $z$ denote the vertex of degree $2$ in $B$ (cf.\ Figure~\ref{fig:subgraphs}).
Furthermore, let $a_1$ and $a_2$ denote the other two vertices of $B$ such that $a_ib_i \in E(B)$ for every $i \in \{ 1, 2\}$.
Since the horns of $B$ have a common neighbour in $G-B$, let us fix such a common neighbour $c$ of $b_1$ and $b_2$ in $G-B$.
As $G$ is infinite, locally finite and connected, there exists some distance increasing ray $R =r_0r_1 \ldots $ w.r.t.\ ${V(B) \cup \{ c \}}$
in~$G$.
We shall distinguish four possible cases where $R$ might start, and derive contradictions for each case.

\setcounter{case}{0}
\begin{case}
$r_0 = z$.
\end{case}

For this case note first that $r_1b_i \notin E(G)$ for every $i \in \{ 1, 2\}$ as otherwise $G[r_1, z, r_2, b_i]$ would be an induced claw.
Next let us verify that $r_1c \notin E(G)$.
Suppose for a contradiction the edge $r_1c$ exists.
Then $G[c, b_1, b_2, r_1]$ is an induced claw as $b_1b_2 \notin E(G)$ and $r_1b_i \notin E(G)$ by the argument above; a contradiction.
In particular, this implies $c \neq r_1$.

Furthermore, we can assume without loss of generality that $G[r_i, a_1, a_2, b_1, b_2]$ is not an induced bull for every $i \geq 1$.
To see this note that $r_i \notin N(V(B))$ for every $i \geq 2$ as $R$ is distance increasing w.r.t.\ ${V(B) \cup \{ c \}}$.
So if $B' = G[r_1, a_1, a_2, b_1, b_2]$ is an induced bull, then $r_1r_2 \ldots$ is a distance increasing ray w.r.t.\ $V(B') \cup \{ c \}$ which starts at the vertex of degree $2$ of $B'$ as well.
Then we would consider $B'$ instead of $B$.
By the previous argument we know that not both of the edges $r_1a_1$ and $r_1a_2$ can exist.
Now suppose for a contradiction that only one of these edges exists, say $r_1a_1$.
Then $G[a_2, r_1, a_1, a_2]$ is an induced claw, since $a_1b_2 \notin E(G)$ as $B$ is an induced bull, $r_1b_2 \notin E(G)$ by the argument above and $r_1a_1 \notin E(G)$ as this would force $G[r_1, a_1, a_2, b_1, b_2]$ to be an induced bull.
Since the analysis for the edge $r_1a_2$ works analogously, we know that $r_1a_1, r_1a_2 \notin E(G)$.

Now we can conclude that $B'' = G[r_1, a_1, z, b_1, a_2]$ is an induced bull with horns $r_1$ and~$b_1$.
So there exists some $c' \in V(G-B'')$ which is a common neighbour of $r_1$ and $b_1$.
As $r_1$ is neither adjacent to $c$ nor to $b_2$, we know that $c' \neq c, b_2$ and since $R$ is distance increasing w.r.t.\ ${V(B) \cup \{ c \}}$ we get $c' \neq r_i$ for all $i \in \mathbb{N}$.
Because $G[r_1, r_2, z, c']$ is not an induced claw and $zr_2 \notin E(G)$, there are two options how this can be avoided:
\begin{equation}\label{c'r_2}
    c'r_2 \in E(G).
\end{equation}
and 
\begin{equation}\label{zc'}
    zc' \in E(G)
\end{equation}

Note that \ref{zc'} and \ref{c'r_2} cannot both hold because then $G[c', b_1, z, r_1]$ would be an induced claw.

Let us first deal with the case that \ref{c'r_2} holds.
Consider the graph $B_1 = G[c', r_1, r_2, r_3, z]$.
If $B_1$ is not an induced bull, then this can only happen because the edge $zc'$ exists, which means~\ref{zc'} holds as well; a contradiction.
So $B_1 = G[c', r_1, r_2, r_3, z]$ is an induced bull.
Then, however, its horns $z$ and $r_3$ would need to have a common neighbour contradicting the property of $R$ being distance increasing w.r.t.\ ${V(B) \cup \{ c \}}$.

So we are left with the situation where~\ref{zc'} holds.
Consider the $B_2 = G[c', z, r_1, r_2, b_1]$.
The only edge which can prevent $B_2$ from being an induced bull would be $c'r_2$, which cannot exist because~\ref{zc'} holds.
So $B_2$ is an induced bull, whose horns $b_1$ and $r_2$ need to have a common neighbour $c'' \in V(G-B_2)$.
Now consider $G[r_1, r_2, r_3, c'']$, which is not allowed to be an induced claw.
Hence, the edge $c''r_1$ exists.
Finally consider $B''' = G[z, r_1, r_2, r_3, c'']$, which cannot be an induced bull since then $z$ and $r_3$ would need to have a common neighbour contradicting the property of $R$ being distance increasing w.r.t.\ ${V(B) \cup \{ c \}}$.
But the only edge which can prevent $B'''$ from being an induced bull is $zc''$.
This leads to the contradiction that $G[b_1, c'', z, r_2]$ is an induced claw and completes Case~1.

\begin{case}
$r_0 = a_i$ for some $i \in \{ 1, 2 \}$.
\end{case}

Say, without loss of generality, $r_0 = a_2$ and consider $G[z, a_2, b_2, r_1]$, which cannot be an induced claw.
The edge $zb_2$ cannot exist because $B$ is an induced bull.
The edge $zr_1$ cannot exist because then $zr_{1}r_{2} \ldots$ would be a distance increasing ray w.r.t.\ ${V(B) \cup \{ c \}}$ starting in $z$, which leads to a contradiction as in Case~1.
Hence, the edge $b_2r_1$ needs to exist.
Now consider $B_2 = G[z, a_2, b_2, r_1, r_2]$, which is an induced bull.
So the horns $z$ and $r_2$ of $B_2$ have a common neighbour $c' \in V(G-B_2)$.
If $c' \in V(B)$, then we get a contradiction to $R$ being distance increasing w.r.t.\ ${V(B) \cup \{ c \}}$.
Otherwise, however, we obtain a distance increasing ray $zc'r_2r_3 \ldots$ w.r.t.\ ${V(B) \cup \{ c \}}$ which starts in $z$.
This leads to a contradiction as in Case~1.
So we have completed our consideration of Case~2.

\begin{case}
$r_0 = c$.
\end{case}

Since $b_1b_2 \notin E(G)$ and $G[c, b_1, b_2, r_1]$ is not an induced claw, we know one of the edges $b_1r_1$ or $b_2r_1$ must exist, say without loss of generality $b_2z'$.
If $B_3 = G[r_1, c, b_1, b_2, a_2]$ were be an induced bull, its horns $b_2$ and $a_2$ would need to have a common neighbour $c''$ in $V(G-B_3)$.
Note that $c' \neq r_i$ and $c'r_{i+3} \notin E(G)$ for every $i \in \mathbb{N}$ as $R$ is distance increasing with respect to $V(B) \cup \{ c \}$.
Now, however, $r_1r_2 \ldots$ is a distance increasing ray w.r.t.~ $V(B_3) \cup \{ c' \}$ which starts at $r_1$.
This is the same situation as in Case~1 and, therefore, leads towards a contradiction.

So $B_3$ is no induced bull and only three edges could possibly witness this, namely $r_1b_1$, $r_1a_2$ or $ca_2$.
First, if $r_1b_1 \in E(G)$, then consider $r_1$ instead of $c$ as the common neighbour of $b_1$ and $b_2$ outside of $B$ and $r_1r_2 \ldots$ as the distance increasing ray w.r.t.\ $V(B) \cup \{ r_1 \}$.
Now we are again in the situation of Case~3 but know that $r_2b_1, r_2b_2 \notin E(G)$, implying that $G[r_1, b_1, b_2, r_2]$ is an induced claw; a contradiction.
Hence, we conclude that $r_1b_1 \notin E(G)$.

Second, suppose that $r_1a_2 \in E(G)$.
Then $a_2r_1r_2 \ldots$ would be distance increasing ray w.r.t.\ $V(B) \cup \{ c \}$ starting at $a_2$, which leads to a contradiction as in Case~2.

Third, suppose $ca_2 \in E(G)$, but $r_1b_1, r_1a_2 \notin E(G)$.
Then $G[c, b_1, a_2, r_1]$ is an induced claw.
This contradiction completes the analysis of Case~3.

\begin{case}
$r_0 = b_i$ for some $i \in \{ 1, 2 \}$.
\end{case}

Let, without loss of generality, $r_0 = b_2$ and consider $G[z, a_2, b_2, r_1]$, which cannot be an induced claw.
By Case~2 and Case~3 we know that $r_1a_2, r_1c \notin E(G)$.
So the edge $ca_2$ must exist.
Consider the bull $B'_3 = G[b_2, c, a_2, b_1, z]$.
If $B'_3$ is induced, then $R$ is distance increasing w.r.t.\ ${V(B'_3) \cup \{ a_1 \} =V(B) \cup \{ c \} }$ and $a_1$ is a common neighbour of the horns $b_1$ and $z$ of $B'_3$.
This puts us again in the situation of Case~1 and leads to a contradiction.

So let us finally consider the case that $B'_3$ is not induced.
The only reason for this is the existence of the edge $cz$.
Then, however, $G[c, b_1, b_2, z]$ is an induced claw; a contradiction.
\end{proof}

We continue by showing that every locally finite connected claw-free graph with at least three ends contains a net, and therefore also a bull, as an induced subgraph.

\begin{lemma}\label{lem:2ends:net}
Every locally finite, connected claw-free and net-free graph has at most two ends.
\end{lemma}

\begin{proof}
Suppose for a contradiction that $G$ is a locally finite connected, claw-free, net-free graph with at least three different ends $\omega_1, \omega_2$ and $\omega_3$.
Let $D$ be a geodesic double ray containing an $\omega_2$-ray and an $\omega_3$-ray, which exists due to Lemma~\ref{LemmaGeo}.
Let $S$ be a finite vertex set which is at least in distance $2$ from $D$ and which separates $D$ from $\omega_1$, i.e., every $\omega_1$-ray with start vertex in $D$ meets $S$.
To see that such a vertex set exists, first pick a finite vertex set $S' \subseteq V(G)$ which pairwise separates $\omega_1, \omega_2$ and $\omega_3$.
Hence, only a finite set $F$ of vertices of $V(D)$ is contained in $V(C(S', \omega_1))$.
Now set $S = N_2(S' \cup F) \cap C(S' \cup F, \omega_1)$, which is still a finite set since $S'$ as well as $F$ are finite and $G$ is locally finite.
Furthermore, $S$ now separates $D$ from $\omega_1$ as desired.

Now consider all shortest $S$-$D$-paths.
Among such shortest paths let $P$ be one that meets $D$ closest to $\omega_3$, say at vertex $d$, i.e., there exists no $S$-$D$-path with an endvertex $d' \neq d$ on the $\omega_3$-ray that is contained in $D$ and starts in $d$.
Choosing $d$ in such a maximal way is possible since $S$ is a finite set and due to the local finiteness each distance class starting from $S$ is a finite set as well. 
Let $b$ be the neighbour of $d$ that lies on the $\omega_2$-ray $R_2$ that is contained in $D$ and starts at $d$.
Further denote the neighbour of $b$ on $R_2$ that is different from $d$ by $y$.
Let $c$ be the neighbour of $d$ that lies on the $\omega_3$-ray $R_3$ which is contained in $D$ and starts at $d$.
Finally, let $a$ be the neighbour of $d$ on $P$ and $z$ be the neighbour of $a$ on $P$ different from $d$.

We first note that there is no edge $ac$ since this would yield a path of the same length as $P$ ending closer to $\omega_3$.
Furthermore, there is no edge $bc$, otherwise $D$ would not be geodesic. Since $G[a,b,c,d]$ cannot be an induced claw, we know that $ab \in E(G)$. This situation would look like depicted in Figure~\ref{KeineDreiEnden}.

\begin{figure}[htbp]
\centering
\includegraphics[width=12cm]{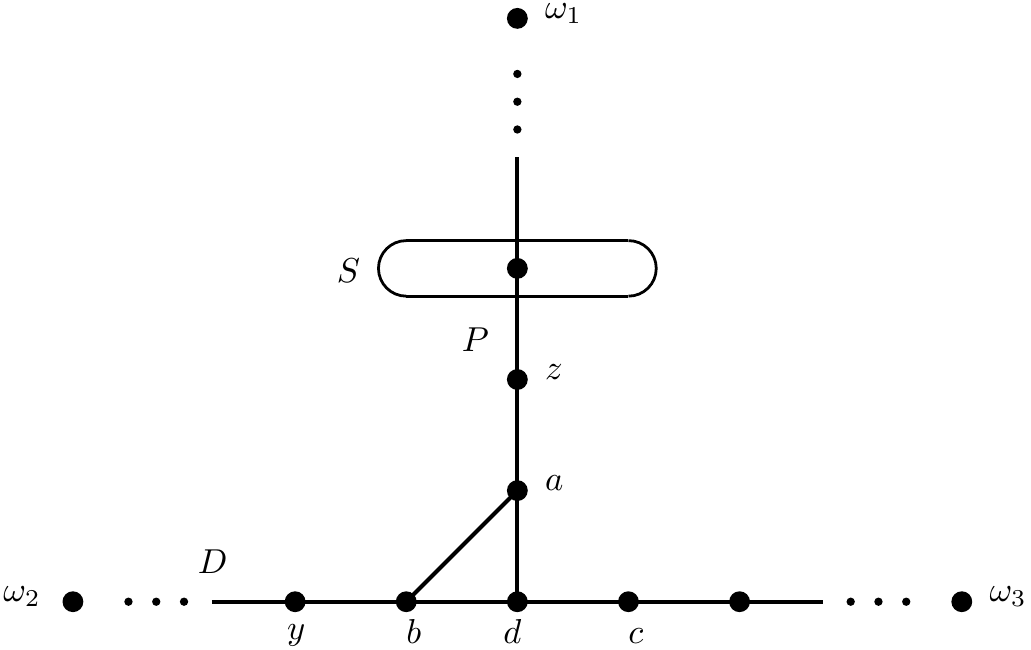}
\caption{The graph $G[a,b,c,d,y,z]$ forms an induced net.}
\label{KeineDreiEnden}
\end{figure}

We furthermore note that $z$ is not adjacent to any vertex in $\{y, b, d, c\}$ since this would yield an $S$-$D$ path shorter than $P$. 
Since $D$ was geodetic, $c$ cannot be adjacent to $y$ or~$b$.
It remains to show that $ay \notin E(G)$.
Suppose for a contradiction that $ay \in E(G)$, then $G[a,d,y,z]$ is an induced claw, contradicting our assumption.
Hence, we proved that $G[a,b,c,d,y,z]$ is an induced net, which contradicts our assumption on $G$.
\end{proof}

Note that the proof of Lemma~\ref{lem:2ends:net} shows also the following.

\begin{corollary}
Let $G$ be a locally finite, connected claw-free and net-free graph with two ends.
Let $D$ be a geodesic double ray containing rays to the two ends of $G$.
Then every vertex of $G$ has distance at most $1$ from $D$. \qed
\end{corollary}

Now let us deduce two corollaries with respect to bulls.

\begin{corollary}\label{lem:2ends:bull}
Every locally finite, connected claw-free and bull-free graph has at most two ends.
\end{corollary}

\begin{proof}
Since every net contains an induced bull, the statement follows immediately from Lemma~\ref{lem:2ends:net}.
\end{proof}

\begin{corollary}\label{lem:2ends:bull-phi}
Let $G$ be a locally finite, connected claw-free.
If the horns of every induced bull $B$ have a common neighbour in $G-B$, then $G$ has at most two ends.
\end{corollary}

\begin{proof}
Since every graph as in the premise of the statement is already bull-free by Theorem~\ref{bullreduct}, the statement follows immediately from Corollary~\ref{lem:2ends:bull}.
\end{proof}

While Lemma~\ref{lem:2ends:net} and its corollaries limit the variety of possible graphs we are considering in terms of the number they can have, we shall now show that these classes are non-trivial.
Before we make explicit constructions, we need to state a definition of which we want to make use of.

Given a graph $G$ and some $k \in \mathbb{N}$, we call a graph $G'$ a \emph{$k$-blow-up} of $G$ if we obtain $G'$ from $G$ by replacing each vertex of $G$ by a clique of size $k$, where two vertices of $G'$ are adjacent if and only if they are either both from a common such clique, or the original corresponding vertices were adjacent in $G$.

\begin{example}
For some $k\in \mathbb{N}$ consider the $k$-blow-up of a ray, yielding a graph with one end, and the $k$-blow-up of a double ray for an example with two ends.
Next we check that these graphs are claw-free.
Suppose for a contradiction there exists an induced claw $C$.
Then the vertex $c$ of degree $3$ in $C$ is contained in a clique corresponding to a vertex $v$ of the ray (or double ray).
Now, however, two non-adjacent neighbours of $c$ could only lie in two cliques which correspond to two neighbours of $v$ vertex of the ray (or double ray).
The third vertex of degree $1$ in $C$ cannot be contained in any clique without causing a contradiction to $C$ being an induced subgraph.

These graphs are also bull-free and, therefore, net-free as well.
Suppose there exists an induced bull $B$, consisting of a triangle $K = G[a_1,a_2,z]$ and the two horns $b_1$ and $b_2$ where $a_ib_i \in E(B)$ for every $i \in \{ 1, 2 \}$.
First note that one edge $xy$ of $K$ has to lie in clique corresponding to a vertex of the ray (or double ray).
By the structure of a bull, $x$ or $y$ is adjacent to a horn $h$ of $B$.
Then, however, by the structure of the whole graph, $h$ is must be adjacent to both, $x$ and $y$, contradicting that $B$ is an induced subgraph.
\end{example}

Finally, let us mention that Lemma~\ref{lem:2ends:net} only holds for locally finite graphs.
In the following example we state non-locally finite, but still countable graphs which are claw-free and net-free, but have $k \geq 3$ or countably many ends:

\begin{example}
Fix for every $i \in \mathbb{Z}$ a clique $K^i_{\aleph_0}$ of size $\aleph_0$ such that the vertex sets of all these cliques and the set $\mathbb{Z}$ are pairwise disjoint.
Furthermore, fix two distinct vertices $i^+$ and $i^-$ in each $K^i_{\aleph_0}$.
For any set $I \subseteq \mathbb{Z}$ we now define the graph $D_I$ as follows.
Set $V(D_I) = (\mathbb{Z} \setminus I) \cup \bigcup_{i \in I} V(K^i_{\aleph_0})$.
If $i, (i+1) \in \mathbb{Z} \setminus I$, set $i(i+1) \in E(D_I)$.
For $i \in I$ and $(i+1) \notin I$, set $i^+(i+1) \in E(D_I)$, and for $i \in I$ and $i-1 \notin I$, set $i^-(i-1) \in E(D_I)$.
Finally, for $i, (i+1) \in I$ set $i^+i^- \in E(D_I)$.
This completes the definition of $D_I$.

It is easy to see that that $D_I$ has precisely $|I| + 2 \in \mathbb{N} \cup \{ \infty \}$ many ends.
Now suppose $D_I$ contains a claw or a net, call it $H$.
Note that each vertex of $H$ with degree $3$ in $H$ has to lie in some $K^i_{\aleph_0}$ with $i \in I$.
From this point on it is easy to see that no matter where the vertices of degree $1$ in $H$ are located in $D_I$, the graph $H$ cannot be an induced subgraph of $D_I$.
\end{example}

\section{Hamiltonicity results}\label{sec:Shep}

The main goal of this section is to extend Theorem~\ref{thm:finShep2.9} to locally finite graphs.
Let us briefly restate the theorem here.

\finShep*

Next we prove Theorem~\ref{infShep2.9.ii}, which is an extension of statement~(\ref{thm:finShep2.9.2}) of Theorem~\ref{thm:finShep2.9} to locally finite graphs.
Our key tool to prove this result is the characterisation of infinite locally finite claw-free and net-free graphs in terms of distance-$2$-completeness, which is Theorem~\ref{thm:infShep2.1}.

\infShepTwo*

\begin{proof}
Let $G$ be a graph as in the statement of this theorem.
By statement~(\ref{thm:finShep2.9.2}) of Theorem~\ref{thm:finShep2.9} we can assume $G$ to be infinite.
By Lemma~\ref{lem:2ends:net} we know that $G$ has at most two ends.
We only write the proof of this theorem for the case that $G$ has precisely two ends, say $\omega_1$ and $\omega_2$.
The case that $G$ has only one end works analogously, but is slightly easier.

Let $S \subseteq V(G)$ be any finite minimal vertex separator of $G$, which exists since $G$ is locally finite.
Furthermore, let us fix some $v \in S$.
By Proposition~\ref{2comp} we know that $G-S$ has exactly two components, call them $L$ and $R$.
By Theorem~\ref{thm:infShep2.1} we know that the graph $G - (S \setminus \{v \})$ is distance-$2$-complete centered at $v$.
Within the graph $G - (S \setminus \{v \})$ let $L_i$ and $R_i$ denote the $i$-th distance classes of $v$ in the components $L$ and $R$, respectively.
Let $\ell \in \mathbb{N}$ be the maximum number with the property that $S$ has a neighbour in $L_{\ell}$ or $R_{\ell}$.
Now we define a hierarchy of subgraphs starting with
\[G_0 := G  \Big[  S \cup \bigcup_{i = 1}^{{\ell}} L_i \cup \bigcup_{i = 1}^{{\ell}} R_i  \Big] .\]

Furthermore. we define: 
\[G_{i+1} := G[V(G_i) \cup L_{i+1} \cup R_{i+1}].\]

Note that $G_0 = G_i$ for every $i \in \mathbb{N}$ with $i \leq \ell$, but $G_i \subsetneqq G_{i+1}$ for every $i \in \mathbb{N}$ with $i \geq \ell$ since $G$ is infinite.

Since each $G_i$ is an induced subgraph of $G$, it is claw-free and net-free as well.
Furthermore, each $G_i$ is finite since $G$ is locally finite.
Using that each subgraph $G[L_i]$ and $G[R_i]$ is complete because $G - (S \setminus \{v \})$ is distance-$2$-complete centered at $v$, it follows easily from the definition of $G_0$ and from $G$ being $2$-connected that each $G_i$ is $2$-connected as well.
Hence we get from statement~(\ref{thm:finShep2.9.2}) of Theorem~\ref{thm:finShep2.9} that each $G_i$ contains a Hamilton cycle.

We now prove that there is a Hamilton cycle $C'_n$ in $G_n$ for every $n > \ell$ such that $|E(C'_n) \cap \delta(R_n)| = 2$ and $|E(C'_n) \cap \delta(L_n)| = 2$ holds.
We start by fixing an arbitrary Hamilton cycle $C_n$ of $G_n$ and fix an orientation of $C_n$.
Starting from $v$ this cycle has to meet $R_{n}$ at some point the first time, say in vertex $w_1$, via the edge $v_1 w_1$ with $v_1 \in R_{n-1}$ (cf.  Figure~\ref{Bild31}).
Beginning from $v$, say the first time $C_n$ leaves $R_n$ happens at vertex $w_2$ via the edge $w_2 v_2$ for some $v_2 \in R_{n-1}$.
To define the desired Hamilton cycle $C'_n$ we follow $C_n$ from $v$ till $w_1$ and collect all vertices from $R_n$ ending in $w_2$, which we can do since each $G[R_n]$ complete.
We now return to $v_2$ via the edge $w_2 v_2$.
Next we follow $C_n$ but whenever the cycle goes from $R_{n-1}$ to $R_n$, say via some edge $v_k  w_k$, and comes back from $R_{n}$ to $R_{n-1}$ the next time, say via an edge $v_{k+1} w_{k+1}$, we replace this segment of the cycle by the edge $v_k v_{k+1}$, which exists since $G[R_{n-1}]$ is complete.

\begin{figure}[htbp]
\centering
\includegraphics[width=5cm]{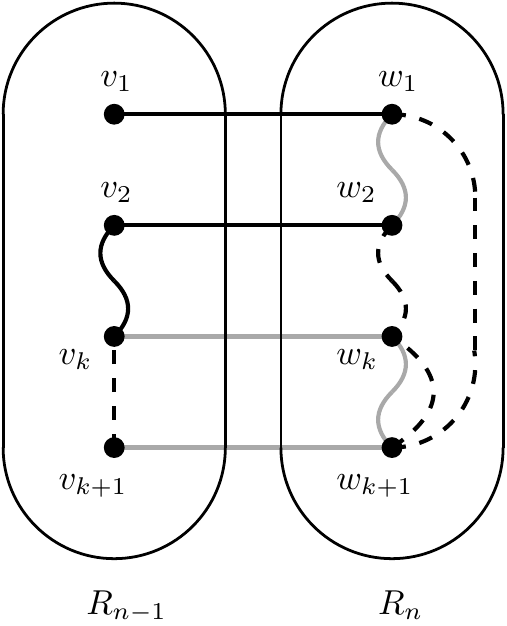}
\caption{Modified cycle $C'_n$ meeting $E(R_{n-1}, R_n)$ only twice by replacing grey edges of $C_n$ by black dashed edges.}
\label{Bild31}
\end{figure}

Similarly, we can modify the cycle to incorporate $L_{n}$ in this manner. Thus we have found the desired Hamilton cycle $C'_n$ in $G_{n}$.
Note that in this process we only altered the initial Hamilton cycle $C_n$ at $L_n$, $L_{n-1}$, $R_{n-1}$ and $R_n$. 

We now show that for any $n > \ell$ the cycle $C'_n$ can be extended without altering edges in $G_{n-1}$ to a Hamilton cycle $D_{n+1}$ of $G_{n+1}$ such that $|\delta(R_n)\cap E(D_{n+1})| = 4$ and $|\delta(R_{n+1})\cap E(D_{n+1})| = 2$ holds as well as analogue statements for $L_n$ and $L_{n+1}$.
We shall give the argument only for $R_{n+1}$ as the modification for $L_{n+1}$ works analogously.
We know that there are at least two independent edges $a_1 b_1, a_2 b_2 \in E(R_n, R_{n+1})$ where $a_i \in R_n$ and $b_i \in R_{n+1}$ for any $i \in \{ 1, 2 \}$, since $G_{n+1}$ is $2$-connected.
If $|R_n| = 2$, we can easily use these two edges to get the desired extension $D_{n+1}$ of $C'_n$.
So we may assume that $R_{n}$ has at least $3$ vertices.
Say without loss of generality that $a_1$ lies before $a_2$ on $C'_n$.
Furthermore, say that $C'_n$ meets $R_n$ the first time (starting from $v$) in $w_1$ via the edge $v_1 w_1$ with $v_1 \in R_{n-1}$ and leaves $R_n$ the last time from $w_2$ via the edge $w_2 v_2$ where $v_2 \in R_{n-1}$. 
We now have to consider two cases.

\setcounter{case}{0}
\begin{case}
$\vert \{ w_1, a_1, a_2, w_2\}\vert \geq 3$.
\end{case}

Without loss of generality let $a_1 \neq w_1$.
In this case we follow $C'_n$ until $w_1$, then collect all vertices from $R_n$ but $w_2$ and $a_2$ such that we end in $a_1$ (see Figure~\ref{fig:case_1}), which we can do since $G[R_{n}]$ is a clique.
Next we use the edge $a_1b_1$, collect all vertices in $R_{n+1}$ while ending in $b_2$, return to $R_n$ via the edge $b_2a_2$.
If $a_2 = w_2$, we can immediately follow $C'_n$ to close a cycle.
Otherwise we use the edge $a_2w_2$ and then proceed with $C'_n$ to to close a cycle.
Doing the same with $L_{n+1}$ yields the desired $D_{n+1}$.
This completes the argument in Case~1.

\begin{figure}[htbp]
\centering
\includegraphics[width=5cm]{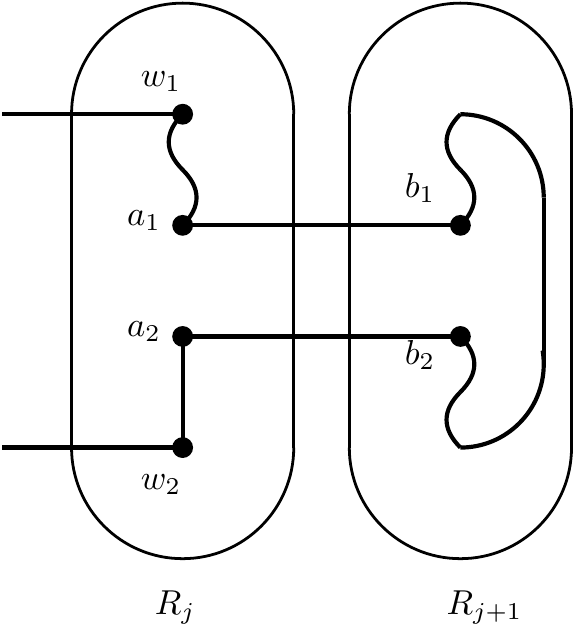}
\caption{The situation in Case~1.}
\label{fig:case_1}
\end{figure}

\begin{case}
$a_1 = w_1$ and $a_2 = w_2$.
\end{case}

Let $x$ be an arbitrary vertex from $R_n - \{ a_1, b_1\}$.
Since $G[v_1 , w_1 , x , b_1]$ is not an induced claw, one of the edges $v_1 x$ or $x b_1$ exists (cf.~Figure~\ref{fig:case_2}).
If $v_1 x$ exists, we can operate as in Case 1 by just switching the roles of $x$ and $w_1$.
Should $x b_1$ exist, we can proceed as in Case 1 as well, this time by switching the roles of $x$ and $a_1$.
This completes the argument for Case~2.

\begin{figure}[htbp]
\centering
\includegraphics[width=6cm]{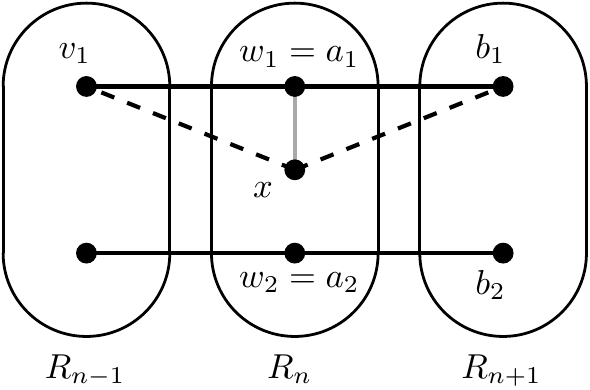}
\caption{The situation in Case~2: the dashed edges prevent the graph $G[v_1 , w_1 , x , b_1]$ from being an induced claw.}
\label{fig:case_2}
\end{figure}

This shows that we can always extend $C'_n$ to the desired cycle $D_{n+1}$.
Since $D_{n+1}$ is also a valid candidate for $C'_{n+1}$ and $D_{n+1} \cap G_{n-1} =  C'_n \cap G_{n-1}$, we can inductively extend $C'_n$ through all $R_m$ and $L_m$ with $m > n$ and obtain a well-defined subgraph $C$ as limit object via its edge set:
\[ E(C) := \left \lbrace e \in E(G) \; ; \; \exists k \in \mathbb{N} \; \colon \; e \in \bigcap^{\infty}_{i \geq k} E(D_{i}) \right \rbrace .\]

The rest of this proof consist of verifying that $\overline(C)$ is a Hamilton circle of $G$.
By the definition of $C$ we immediately get that every vertex of $G$ is contained in $C$.
It remains to check that $\overline{C}$ is a circle in $\vert G \vert$.
From the definition of all the $D_i$ and $C$ we immediately get that every vertex of $G$ has degree $2$ in $C$.
By Lemma~\ref{circ} it remains to prove that $C$ is topologically connected and that every end of $G$ has degree $2$ in~$\overline{C}$.

In order to prove that $\overline{C}$ is topologically connected, it is enough to show that $C$ meets every finite cut of $G$ by Lemma~\ref{top-con}.
This, however, holds since each finite cut $F$ of $G$ is eventually contained in $G_m$ for all $m > m_0$ where $m_0$ is some sufficiently large integer, which means that each Hamilton cycle $D_m$ of $G_m$ for $m > m_0$ meets $F$ in the same set of edges and, hence, so does $C$.
We even get that each finite cut $F$ of $G$ is met in an even number of edges by $C$, since the intersection of a cycle and a cut is always even.

Since being topologically connected and being arc-connected is equivalent for closed subspaces of $|G|$ by Lemma~\ref{arc_conn}, we know that for every end $\omega$ of $G$ there exists at least one arc in $\overline{C}$ with $\omega$ as its endpoint.
So each end of $G$ already has degree at least $1$ in $\overline{C}$.

Next let us prove that each end of $G$ has degree at most $2$ in $\overline{C}$.
For this let us define $R_{\geq n} = \bigcup^{\infty}_{i=n}R_i$ and $L_{\geq n} = \bigcup^{\infty}_{i=n}L_i$ for every $n > \ell$.
Since each $R_n$ and each $L_n$ separates the two ends of $G$ if $n > \ell$, say every $\omega_1$-ray has a tail in $R_{\geq n}$ and every $\omega_2$-ray has a tail in $L_{\geq n}$ for every $n > \ell$.
By definition of $C$ we get that ${|\delta(R_{\geq n}) \cap E(C)| = |E(R_{n-1}, R_{n}) \cap E(C)| = 2}$ holds for every $n > \ell$.
Due to Lemma~\ref{jumping-arc} we now know that $\overline{C}$ cannot contain three vertex disjoint arcs all of which have $\omega_1$ as their endpoint.
So $\omega_1$ has degree at most $2$ in $\overline{C}$.
An analogue argument shows that $\omega_2$ has degree at most $2$ in $\overline{C}$ as well.

Finally, we prove that each end of $G$ has degree at least $2$ in $\overline{C}$.
Theorem~\ref{cycspace} tells us that each end of $G$ has an even degree in $\overline{C}$ if $C$ meets every finite cut of $G$ in a even number of edges.
This holds as already proven above.
So we can conclude that both ends of $G$ have degree precisely $2$ in $\overline{C}$, which completes the proof that $\overline{C}$ is a Hamilton circle~of~$G$.
\end{proof}

We move on by proving Theorem~\ref{infShep2.9.iii}, which is an extension of statement~(\ref{thm:finShep2.9.3}) of Theorem~\ref{thm:finShep2.9} to locally finite graphs.
Recall that we call a locally finite graph $G$ \emph{topologically $k$-leaf-connected} where $k \in \mathbb{N}$ if $|V(G)| > k$ and given any set $S \subseteq V(G) \cup \Omega(G)$ with $|S| = k$, then $G$ has a topological spanning tree $\overline{T}$ whose set of leaves is precisely $S$.

\infShepThree*

\begin{proof}
By statement~(\ref{thm:finShep2.9.3}) of Theorem~\ref{thm:finShep2.9} we might assume for both implications that $G$ is infinite.
Let us first assume that $G$ is an infinite, but locally finite graph that is topologically $k$-leaf-connected.
We show that $G$ is $(k+1)$-connected.
Assume for a contradiction that $G$ has a vertex separator $S \subseteq V(G)$ of size at most $k$.
Let $S' \supseteq S$ be a superset of $S$ such that $S'$ still separates $G$, which is possible since $G$ is infinite, and $|S'| = k$.
By the topologically $k$-leaf-connectedness there exists a subgraph $T$ of $G$ such that $\overline{T}$ is a topological spanning tree of $G$, whose set of leaves is exactly $S'$.
Since $\overline{T}$ contains vertices from two components of $G-S'$, it must pass $S$ by Lemma~\ref{jumping-arc}, yielding a vertex of from $S$ that has degree at least $2$ in $T$; a contradiction.

Suppose for the other implication that $G$ is an infinite, but locally finite claw-free and net-free graph which is $(k+1)$-connected for $k \geq 2$.
We show that $G$ is topologically $k$-leaf-connected.
By Theorem~\ref{thm:infShep2.1} we know that for every finite minimal vertex separator $S \subseteq V(G)$ of $G$ and every $v \in S$, the graph $G - (S \setminus \{ v\})$ is distance-$2$-complete centered at $v$.
So let us fix such an $S$ and some $v \in S$.
Let $G_i, R, R_i, L$ and $L_i$ be defined as in the proof of Theorem~\ref{infShep2.9.ii}. 

Let us fix some $B = \{l_1, \ldots, l_k \} \subseteq V(G) \cup \Omega(G)$  for the rest of the proof.
We have to show that a topological spanning tree of $G$ exists whose set of leaves is exactly $B$.
By Theorem~\ref{thm:infShep2.1} we know that $G$ has at most two ends.
We shall give the proof only in the case that $G$ has precisely one end.
For the case that $G$ contains two ends the argument can easily be adapted.
A consequence of assuming $G$ to have only one end is that $L$ or $R$ is finite, say $L$.
We shall distinguish two cases, namely whether $B$ contains one or the two ends of $G$ or not.

First let us assume that $B$ contains no ends of $G$.
Let $\ell \in \mathbb{N}$ be sufficiently large such that $B, L$ and a finite connected subgraph in $R$ containing $N(S) \cap R$ is contained in $G_{\ell-3}$.
Similar to our proof of Theorem~\ref{infShep2.9.ii} we now show that there exists a finite spanning tree $T_{\ell + 1}$ of $G_{\ell + 1}$ with precisely $B$ as its set of leaves such $|\delta(R_{\ell + 1})\cap E(T_{\ell+1})| = 2$ holds in~$G_{\ell + 1}$.

To prove this, we first verify that $G_{\ell + 1}$ is also $(k+1)$-connected.
Otherwise, there would exist a separator $S_k$ of size at most $k$, separating two vertices $x$ and $y$ in $G_{\ell + 1}$, but not in~$G$.
Hence, there exists an $x$--$y$-path $P$ in $G$ disjoint to $S_k$.
Since $P$ does not exist in $G_{\ell + 1}$, it must pass through $R_{\ell + 1}$ to $R_{\ell + 2}$.
By shorten $P$ on $G[R_{\ell + 1}]$, which is a clique, we obtain an $x$--$y$-path in $G_{\ell + 1}$ which is disjoint to $S_k$; a contradiction.

Since $G_{\ell + 1}$ is $(k+1)$-connected and, as induced subgraph of $G$, also claw-free and net-free, there exists a spanning tree $T'_{\ell + 1}$ of $G_{\ell + 1}$ whose set of leaves is $B$ by statement~(\ref{thm:finShep2.9.3}) of Theorem~\ref{thm:finShep2.9}.
Next we modify this tree to obtain the desired tree $T_{\ell + 1}$.
First we root $T'_{\ell + 1}$ in $v$ and orient its edges away from the root.
Now we get $T_{\ell + 1}$ from $T'_{\ell + 1}$ by shortening all but one of the directed paths $P$ starting and ending in $R_{\ell}$ and otherwise only using vertices from $R_{\ell + 1}$ by an edge from the start vertex to the endvertex of $P$, which exists since $G_{\ell + 1}$ is a clique.
Note that if we do not need to do this replacement, then $T'_{\ell + 1}$ is already as desired.
We now modify the remaining of such paths in $R_{\ell + 1}$ to one containing all vertices of $R_{\ell + 1}$, but with the same start and endvertex.
The resulting graph is our $T_{\ell + 1}$, which is indeed a tree since it is connected and every cycle in $T_{\ell + 1}$ would yield a cycle in $T'_{\ell + 1}$ either directly or by replacing edges from $E(T_{\ell + 1}) \setminus E(T'_{\ell + 1})$ by the corresponding paths in $R_{\ell + 1}$.

We can now extend $T_{\ell + 1}$ to a topological spanning tree $\overline{T}$ of $G$, where $T$ is a corresponding subgraph of $G$.
For this we extend one branch of $T_{\ell + 1}$ starting in $v$ and ending in some leaf $l_i$ of $T_{\ell + 1}$ that contains edges from~$G[R_{\ell + 1}]$.
We modify such a branch at an edge in $G[R_{\ell + 1}]$ to an arc via the end of $G$ starting in $v$, ending in $l_i$ and containing all remaining vertices of $G$, which are precisely those in $\bigcup_{i = \ell + 2}^{\infty} R_i$.
This modification can be done similarly to our extension of the Hamilton cycles in the proof of Theorem~\ref{infShep2.9.ii}.
To see that the resulting standard subspace $\overline{T}$ of $G$ is indeed a topological spanning tree of $G$, we have to check that it is topologically connected and does not contain a circle from $|G|$.
Similarly as in the proof of Theorem~\ref{infShep2.9.ii}, it is easy to check that $T$ intersects every finite cut of~$G$.
Hence, Lemma~\ref{top-con} implies that $\overline{T}$ is topologically connected.
To see that $\overline{T}$ dos not contain any circle from $|G|$, note that any such circle would also induce a cycle in some $T_{n}$ for sufficiently large $n \in \mathbb{N}$, contradicting that $T_{n}$ is a tree.
This completes the proof for the first case.

Now let us assume that $B = \{l_1, \ldots, l_{k-1}, \omega \}$ contains the end $\omega$ of $G$.
As in the first case choose a sufficiently large $\ell \in \mathbb{N}$ such that $B - \{\omega\}$, $L$ and a finite connected subgraph in $R$ containing $N(S) \cap R$ is contained in $G_{\ell - 3}$.
Let $w$ be a vertex in $R_{\ell}$ with a neighbour in~$R_{\ell + 1}$.
By statement~(\ref{thm:finShep2.9.3}) of Theorem~\ref{thm:finShep2.9} there is a spanning tree $T_{\ell}$ in $G_{\ell}$ with $B - \{\omega\} \cup \{w \}$ as its set of leaves.
Let us pick $v \in S$ as the root of $T_{\ell}$.
But now we can extend the branch ending in $w$ by an $\omega$-ray $R$ starting in $w$ with $V(R) = \{w\} \cup \bigcup_{i = \ell + 1}^{\infty} R_i$.
Similar as in the first case it is easy to verify that the closure $\overline{T}$ of the resulting subgraph $T$ yields a topological spanning tree of $G$ whose set of leaves is precisely $B$.
\end{proof}

Finally, we prove Theorem~\ref{infShep2.9.i}, which forms an extension of statement~(\ref{thm:finShep2.9.1}) of Theorem~\ref{thm:finShep2.9} to locally finite graphs.

\infShepOne*

\begin{proof}
We shall distinguish three cases with respect to the connectivity of $G$.
Suppose first that $G$ is not $2$-connected.
Then let $v$ be a cut vertex of $G$.
By Theorem~\ref{thm:infShep2.1} we know that $G$ is distance-$2$-complete centered at $v$.
So there are precisely two components of $G-v$, both of which are infinite if $G$ has two ends, and just one of them is infinite in case $G$ has only one end.
Using the structure of distance-$2$-complete graphs we easily find either a spanning double ray if $G$ has two ends, or a spanning ray if $G$ has only one end.

For the second case let us assume that $G$ is $3$-connected.
By Theorem~\ref{infShep2.9.iii} we know that $G$ is $2$-leaf-connected.
So we can find for any two distinct $x, y \in V(G) \cup \Omega(G)$ a Hamilton arc of $G$ with $x$ and $y$ as endpoints.
In the case that $G$ has precisely two ends, we can find a Hamilton arc of $G$ with these ends as its endpoints.
Since $G$ has no further ends, we immediately get that this Hamilton arc induces a double ray in $G$.

So let us focus on the case that $G$ has just one end.
We fix some finite minimal vertex separator $S \subseteq V(G)$ and some $v \in S$.
Let $G_i, R, R_i, L$ be defined as in the proof of Theorem~\ref{infShep2.9.ii} and, without loss of generality, say that $L$ is finite. 
Furthermore, let $\ell \in \mathbb{N}$ be sufficiently large such that $L$ and a finite connected subgraph in $R$ containing $N(S) \cap R$ is contained in $G_{\ell-3}$.
Since $G_{\ell + 1}$ is also $3$-connected as argued within the proof of Theorem~\ref{infShep2.9.iii}, we can find a Hamilton path $P$ in $G_{\ell + 1}$ whose start vertex is some arbitrary $x \in V(G_{\ell + 1})$ and whose endvertex is some $y \in R_{\ell + 1}$ that has a neighbour in $R_{\ell + 2}$ in $G$.
Again by making use of $G$ being distance-$2$-complete, we can easily extend $P$ through all the $R_i$ for $i \geq \ell + 2$ yielding a spanning ray of $G$ that starts at $x$.
This completes the second case.

It remains to prove the statement under the assumption that $G$ is $2$-connected, but not $3$-connected.
Hence, there is a minimal separator $\{u, v\}$ of $G$ with $u$ and $v$ being distinct. 
By Theorem~\ref{thm:infShep2.1} we know that $G - u$ and $G - v$ are distance-$2$-complete centered at $v$ and~$u$, respectively.
Let $R$ and $L$ be the two components of $G - \{ u, v \}$.

Since $\{u, v\}$ is a minimal vertex separator of $G$, we know that $u$ has at least one neighbour in $R$ as well as in $L$.
Furthermore, any two neighbours of $u$ in $R$ (or $L$) must be adjacent due to the claw-freeness of $G$.
Hence, any two neighbours of $u$ in $R$ lie in some common distance class of $v$ within $R$ or in two successive distance classes of $v$ within $R$.
An analogue statement holds for neighbours in $L$.

Now suppose $u$ has two distinct neighbours $u_1$ and $u_2$ in $L$ or $R$, say $R$.
Using that $G-u$ is distance-$2$-complete centered at $v$, we can find a spanning ray or double ray in $G-u$, depending whether $G$ has only one or two ends, that uses the edge $u_1u_2$.
Similarly as before, we incorporate $u$ by replacing the edge $u_1u_2$ by the path $u_1uu_2$.

So we may also assume that $u$ has precisely one neighbour in each of $R$ and $L$.
Since $G$ has at least one end, one of $L$ or $R$ must be infinite, say $R$.
Since $R$ is infinite and $G$ is $2$-connected, $N(v) \cap R$ cannot consist of just one vertex.
Hence $v$ has at least two neighbours in $N(v) \cap R$.
By changing the roles of $u$ and $v$ we are now done due to our earlier observation for the case that $u$ has two neighbours in $R$.
\end{proof}

\section*{Acknowledgements}
Karl Heuer was supported by a postdoc fellowship of the German Academic Exchange Service (DAAD) and by the European Research Council (ERC) under the European Union's Horizon 2020 research and innovation programme (ERC consolidator grant DISTRUCT, agreement No.\ 648527).

Deniz Sarikaya is thankful for the financial and ideal support of the Studienstiftung des deutschen Volkes and the Claussen-Simon-Stiftung.

Furthermore, both authors would like to thank Max Pitz for a comprehensive feedback on an early draft of this paper.
Also they would like to thank Hendrik Niehaus and J.~Pascal Gollin for helpful comments on an early version of this article.

\end{document}